\documentclass[11pt]{article}
\usepackage[utf8]{inputenc}

\usepackage{fullpage,times,url,natbib}

\usepackage{amsthm,amsfonts,amsmath,amssymb,epsfig,color,float,graphicx,verbatim}
\usepackage{algorithm,algorithmic}

\usepackage{hyperref}
\hypersetup{
	colorlinks   = true, 
	urlcolor     = blue, 
	linkcolor    = blue, 
	citecolor   = black 
}

\newcommand{\reals}{\mathbb{R}}

\newcommand{\norm}[1]{\|#1\|}

\newcommand{\lemref}[1]{Lemma~\ref{#1}}

\renewcommand{\eqref}[1]{Eq.~(\ref{#1})}

\newtheorem{theorem}{Theorem}
\newtheorem{lemma}[theorem]{Lemma}
\newtheorem{assumption}{Assumption}
\newtheorem{proposition}{Proposition}
\newtheorem*{remark}{Remark}

\title{High-Order Oracle Complexity \\of Smooth and Strongly Convex Optimization}
\author{Guy Kornowski \qquad Ohad Shamir\\Weizmann Institute of Science}

\begin{document}
\maketitle

\begin{abstract}
    In this note, we consider the complexity of optimizing a highly smooth (Lipschitz $k$-th order derivative) and strongly convex function, via calls to a $k$-th order oracle which returns the value and first $k$ derivatives of the function at a given point, and where the dimension is unrestricted. Extending the techniques introduced in \citet{arjevani2019oracle}, we prove that the worst-case oracle complexity for any fixed $k$ to optimize the function up to accuracy $\epsilon$ is on the order of $\left(\frac{\mu_k D^{k-1}}{\lambda}\right)^{\frac{2}{3k+1}}+\log\log\left(\frac{1}{\epsilon}\right)$ (in sufficiently high dimension, and up to log factors independent of $\epsilon$), where $\mu_k$ is the Lipschitz constant of the $k$-th derivative, $D$ is the initial distance to the optimum, and $\lambda$ is the strong convexity parameter. 
\end{abstract}

\section{Introduction}

The complexity of optimizing  functions of a given class using iterative methods is a fundamental question in the theory of optimization. A standard way to approach this is via oracle-based complexity (cf. \citet{YudNem83}): Given local access to the function's values and derivatives at various points, how many such points do we need to sequentially query in order to optimize any function of a given class to some target accuracy $\epsilon$? This forms a good model for generic, unstructured optimization problems, where most practical methods are iterative and rely on such local information. 

Classical results on oracle complexity focus mostly on zero-order and first-order oracles, which provide information about the function's values or gradients, and thus capture algorithms which rely on such local information, such as gradient descent and accelerated gradient descent \citep{YudNem83,nemirovski2005efficient,nesterov2018lectures}. However, there has been much progress in recent years in understanding methods which rely on Hessians and higher-order derivatives, both in terms of new methods and complexity upper bounds (e.g., \citet{nesterov2006cubic,nesterov2008accelerating,baes2009estimate,bubeck2019near,jiang2019optimal,gasnikov2019near,kamzolov2020near,nesterov2020inexact,nesterov2020superfast}), as well as complexity lower bounds \citep{arjevani2019oracle,agarwal2018lower}.

In this note, we focus on a particularly well-behaved class of functions on $\reals^d$: Those which are both highly-smooth (with a globally $\mu_k$-Lipschitz $k$-th order derivative, for some fixed $k\geq2,\mu_k$), and strongly convex with parameter $\lambda$\footnote{In our case of twice differentiable functions, this can be defined as $\nabla^2f(x)\succeq\lambda I$ for all $x\in\reals^d$.}. Generalizing and extending the techniques introduced in \citep{arjevani2019oracle}, we formally prove an oracle complexity lower bound of 
\[
\Omega\left(\left(\frac{\mu_k D^{k-1}}{\lambda}\right)^{\frac{2}{3k+1}}+\log\log\left(\frac{1}{\epsilon}\right)\right)
\]
for any sufficiently small $\epsilon$, where $D$ is the initial distance from the global optimum. This lower bound holds for any deterministic algorithm\footnote{This captures standard algorithms for this setting, and is mostly for simplicity. Indeed, the proof can be extended to any randomized algorithm using, for example, the techniques of \cite{woodworth2017lower}, which performed a similar extension for first-order oracles, at the cost of a considerably more involved proof.}. Moreover, under a mild assumption on the size of the derivatives at the optimum, we show that the lower bound above can be attained (up to logarithmic factors independent of $\epsilon$), using a combination of the accelerated Taylor descent algorithm of \citet{bubeck2019near} and the cubic regularized Newton algorithm from \citet{nesterov2008accelerating}. 

The results generalize those of \citet{arjevani2019oracle} (which in the strongly convex case considered only second-order oracles), and provide qualitatively similar conclusions. For example, the bound implies that even though the convergence order is eventually quadratic (as captured by the $\log\log(1/\epsilon)$ term, growing extremely slowly in $\epsilon$), the optimization complexity is largely influenced by geometry-dependent factors such as the initial distance $D$ and the Lipschitz/strong convexity parameters $\mu_k$ and $\lambda$ (although the latter two are attenuated as $k$ increases, due to the $\frac{2}{3k+1}$ power term). Moreover, the bound implies that higher-order methods (with any $k>1$ and assuming only Lipschitzness of the $k$-th order derivatives) must have a polynomial dependence on $D$, in sharp contrast to first-order methods ($k=1$) which can actually attain logarithmic dependence on $D$, assuming the gradient is Lipschitz. 

It is important to emphasize that while the upper bound holds independently of the dimension, the lower bound is high-dimensional in nature. Namely,
\[
d=\Omega(T)=\Omega\left(\left(\frac{\mu_k D^{k-1}}{\lambda}\right)^{\frac{2}{3k+1}}+\log\log\left(\frac{1}{\epsilon}\right)\right)~.
\]
Not only is this property common practice in the optimization lower bounds literature \citep{YudNem83, nesterov2018lectures}, it is crucial. Indeed, given an initial upper bound on the distance to the optimum, a simple combination of the center of gravity method \citep{bubeck2014convex} and cubic regularized Newton yields an algorithm which requires only $\widetilde{\mathcal{O}}(d+\log\log(1/\epsilon))$ steps. Thus, if $d=o((\frac{\mu_k D^{k-1}}{\lambda})^{\frac{2}{3k+1}})$ our lower bound cannot possibly hold. We address this regime and complement our results by proving for any $d=\mathcal{O}((\frac{\mu_k D^{k-1}}{\lambda})^{\frac{2}{3k+1}})$ an oracle complexity lower bound of $\Omega(d)$. Overall this establishes tight complexity bounds both in the high and low dimensional regimes, up to logarithmic factors.

\section{Main results}

Consider a high-order oracle, which given a point $x\in\mathbb{R}^d$ returns some function’s value and all of its derivatives up to order $k$ evaluated at the point: $f(x),\nabla f(x),\nabla^{2}f(x),...,\nabla^{k}f(x)$. Given access to such an oracle, an algorithm produces a sequence of points $x^{1},x^{2},...,x^{T}$, with each $x^{t}$ being some deterministic function of the oracle’s responses at $x^{1},x^{2},...,x^{t-1}$. The algorithm's goal is to approximate the function's global minimum $x^*$. That is, after some $T$ queries, produce $x^T$ such that $f(x^{T})-f(x^{*})\leq\epsilon$. We will only consider objective functions $f$ which come from the relatively well behaved class of $\lambda$-strongly convex, $k\geq2$ times differentiable functions with a $\mu_{k}$-Lipschitz $k$-th order derivative (with respect to the tensor operator norm).

To state our lower bound, we will make two weak assumptions, which  essentially ensure that all terms in the bound are at least some positive constant, and that we take logarithms of positive quantities:


\begin{assumption} \label{lower bound assumptions}
\begin{enumerate}
    \item(We do not start off too lucky)\\
$
D
>
\max\left\{\sqrt{3}\left(\frac{k!2^{\frac{k+3}{2}}\lambda}{\mu_{k}}\right)^{\frac{1}{k-1}},\sqrt{3}(12)^{\frac{3k+1}{2(k+1)}}\left(\frac{k!2^{\frac{k+3}{2}}\lambda}{\mu_{k}}\right)^{\frac{5}{4(2k-1)}}\right\}
$
\item(Our standards are not too low)
$
\epsilon
<
\min\left\{\left(\frac{(4\cdot k!)^{2}\lambda^{k+1}}{\mu_{k}^{2}}\right)^{\frac{1}{k-1}},
\frac{\lambda}{8}\left(\frac{k!2^{\frac{k+3}{2}}\lambda}{\mu_{k}}\right)^{\frac{2}{k-1}}\right\}
$
\end{enumerate}
\end{assumption}

\begin{theorem} \label{lower bound main result}
For any $k\geq2$ and positive $\lambda,\mu_{k},D,\epsilon$ that satisfy Assumption \ref{lower bound assumptions},
and an algorithm based on a $k$-th order oracle as described above, there exists a function $f:\mathbb{R}^{d}\rightarrow\mathbb{R}$ such that:

\begin{itemize}
    \item 
    $f$ is $\lambda$-strongly convex, $k$ times differentiable with $\mu_{k}$-Lipschitz $k$-th order derivative and has a global minimum $x^{*}$ satisfying $\|x^{1}-x^{*}\|\leq D$.
    \item
    The index $T$ required to ensure $f(x^{T})-f(x^{*})\leq\epsilon$ is at least
    \begin{equation*}
    c\cdot\left(\left(\frac{\mu_{k}D^{k-1}}{3^{\frac{k-1}{2}}2^{\frac{k+3}{2}}k!\lambda}\right)^{\frac{2}{3k+1}}+\log_{k}\log\left(\left(\frac{(4\cdot k!)^{2}\lambda^{k+1}}{\mu_{k}^{2}}\right)^{\frac{1}{k-1}}\cdot\frac{1}{\epsilon}\right)
    \right)
    \end{equation*}
    for some absolute constant $c>0$. Consequently,
    \begin{equation*}
    T \geq c_{k}\left(\left(\frac{\mu_{k}D^{k-1}}{\lambda}\right)^{\frac{2}{3k+1}}+\log\log\left({\frac{\lambda^{\frac{k+1}{k-1}}}{\mu_{k}^{\frac{2}{k-1}}}}\cdot\frac{1}{\epsilon}\right)\right)
    \end{equation*}
    for some constant $c_k>0$ which depends only on $k$.
\end{itemize}
\end{theorem}

In the same setting as above, denote $M:=\max\left\{ \|\nabla^{2}f(x^{*})\|,...,\|\nabla^{k}f(x^{*})\|\right\}$ (under the operator norm).
Once again, for technical reasons we will make two weak assumptions:
\begin{assumption} \label{upper bound assumption}
\begin{enumerate}
\item(We do not start off too lucky) $D>2$
\item(Our standards are not too low) $\epsilon<\frac{\lambda^3}{2eM^2}$ 
\end{enumerate}
\end{assumption}
The following theorem states that the lower bound is essentially tight up to a logarithmic factor.

\begin{theorem} \label{upper bound main result}
Let $f:\mathbb{R}^{d}\to\mathbb{R}$ be a function which satisfies the same assumptions as in Theorem \ref{lower bound main result}. Assume the parameters satisfy Assumption \ref{upper bound assumption} (instead of \ref{lower bound assumptions}). Then there exists a deterministic algorithm that utilizes a $k$-th order oracle of $f$ and produces a sequence $x^{1},x^{2},\dots$ such that $f(x^{T})-f(x^{*})<\epsilon$ where
\begin{equation*}
    T=C_{k}\cdot\widetilde{\mathcal{O}}\left(\left(\frac{\mu_{k}D^{k-1}}{\lambda}\right)^{\frac{2}{3k+1}}+\log\log\left(\frac{\lambda^{3}}{M^{2}\epsilon}\right)\right)
\end{equation*}
for some constant $C_{k}>0$ which depends only on $k$.
\end{theorem}

\begin{remark}
More specifically, the $\widetilde{\mathcal{O}}(\cdot)$ hides the fact that the first summand is multiplied by
$$\log\left(\frac{M^2D^2}{\lambda}\cdot\left(\frac{\mu_{k}}{M}\right)^{\frac{2}{k-2}}\right)$$ Interestingly, note that $M$ appears in the bound only inside logarithmic factors. 
\end{remark}

Lastly, we consider the low dimensional regime. As before, we make the following assumptions:
\begin{assumption} \label{assumption: low dimension}
\begin{enumerate}
    \item(We do not start off too lucky) $D>\sqrt{\frac{\mu_{k}}{k!2^{\frac{k+3}{2}}\lambda}+\sqrt{2}\left(\frac{\mu_{k}}{k!2^{\frac{k+3}{2}}\lambda}\right)^{3/2}}$
    \item(Our standards are not too low) $\epsilon<\frac{\lambda}{32}$
    \item(Low dimensional regime) $d\leq\left(\frac{\mu_k D^{k-1}}{k!2^{\frac{k+3}{2}}\lambda}\right)^{\frac{2}{3k+1}}$
\end{enumerate}
\end{assumption}
\begin{theorem} \label{thm: low dim lower bound}
For any $k\geq2$, positive $\lambda,\mu_{k},D,\epsilon$ and $d\in\mathbb{N}$ that satisfy Assumption \ref{assumption: low dimension},
and an algorithm based on a $k$-th order oracle as described above, there exists a function $f:\mathbb{R}^{d}\rightarrow\mathbb{R}$ such that:
\begin{itemize}
    \item 
    $f$ is $\lambda$-strongly convex, $k$ times differentiable with $\mu_{k}$-Lipschitz $k$-th order derivative and has a global minimum $x^{*}$ satisfying $\|x^{1}-x^{*}\|\leq D$.
    \item
    The index $T$ required to ensure $f(x^{T})-f(x^{*})\leq\epsilon$ is at least $c\cdot{d}$, for some absolute constant $c>0$.
\end{itemize}
\end{theorem}

The proofs of the theorems appear in the next sections, and are a generalization of the techniques used in \citet{arjevani2019oracle} for the convex and 2nd-order oracle cases. In a nutshell, the lower bound is based on a rotation of a strongly convex function of roughly the form
\[
f(x_1,x_2,\ldots)=\sum_{j=1}^{T-1}|x_i-x_{i+1}|^{k+1}-x_1+\frac{\lambda}{2}\norm{x}^2~.
\]
The structure of the function is inspired by Nesterov's quadratic ``worst function in the world'' for first-order oracles \citet{nesterov2018lectures}. Here, higher-order powers and algorithm-dependent rotations are used to ensure that with a $k$-th order oracle, the algorithm is forced to query only on a certain subspace, whose points are relatively far from the global optimum. The main technical difficulty of the proof is in understanding the structure of the global optimum, and proving a lower bound on its distance from the relevant subspace. 
As to the upper bound in Theorem \ref{upper bound main result}, the proof is algorithmic in nature. We start by using a recent high order optimization algorithm called ``Accelerated Taylor Descent'' (ATD) \citep{bubeck2019near} which was designed and analyzed for highly smooth convex functions (not strongly convex ones). We utilize the facts that in the strongly convex case, decreasing the error necessarily decreasing the distance to the optimum, and that the error bounds for ATD decay with the initial distance. Thus, by repeatedly applying and re-starting ATD, we get a linear convergence towards the optimal point (a similar idea was also utilized in \citet{arjevani2019oracle} for the second-order oracle case). Once this gets us close enough to the optimum, we switch to performing cubic regularized Newton (CRN) steps \citep{nesterov2008accelerating}. Naively we would like to gain the quadratic convergence rate that CRN is proved to achieve for strongly convex functions when initialized close enough to the optimum. The difficulty here lies in the fact that CRN achieves this rate only under the additional assumption that the Hessians are globally Lipschitz, whereas we here we assume that the $k$-th order derivatives (for possible $k>2$) are Lipschitz. To circumvent this, we make the additional mild assumption that the derivatives are bounded at the optimum by $M$, and combine this with the $k$-th order Lipschitz assumption to guarantee sufficiently Lipschitz Hessians in a neighborhood of the optimum.

We prove Theorems \ref{lower bound main result} and \ref{upper bound main result} in the next two sections, respectively. The proof of Theorem \ref{thm: low dim lower bound} is similar to the proof of Theorem  \ref{lower bound main result}, so we defer it to Appendix \ref{Appendix: proof of low dim lower bound}.

\section{Proof of Theorem \ref{lower bound main result}} \label{sec: proof of main lower bound}
The proof is constructed of several parts. In subsection \ref{subsec: lower proof 1} we define a parameterized family of functions and prove their respective minima satisfy certain qualities. Then, in subsection \ref{lower proof 2} we introduce yet another, richer, parameterized family of functions and relate it to the earlier class. Afterwards, in subsection \ref{subsec: lower oracle complexity} we provide an oracle complexity lower bound for the class we constructed. Finally, in subsection \ref{lower proof 4} we choose the remaining parameters such that they fix a function $f$ which proves the theorem.

\subsection{Simplified function} \label{subsec: lower proof 1}
Fix $k,d,\widetilde{T}\in \mathbb{N},\ \gamma,\widetilde{\lambda}>0$. Assume $ d>\widetilde{T}\geq\frac{4\gamma}{\widetilde{\lambda}^{\frac{k}{k-1}}}$, $\gamma\geq\max\{\widetilde{\lambda}^{\frac{k}{k-1}},12^{\frac{2k}{k+1}}\widetilde{\lambda}^{\frac{2k}{2k-1}}\}$ and consider the function

\begin{equation*}
    \widetilde{f}_{\gamma}(x_{1},...,x_{d})=
    \frac{1}{k+1} 
    \sum_{i=1}^{\widetilde{T}-1}
    |x_{i}-x_{i+1}|^{k+1}-\gamma x_{1}+
    \frac{\widetilde{\lambda}}{2}\|x\|^{2}~.
\end{equation*}
For the sake of notational simplicity, we will assume $\gamma$ is fixed and denote the function by $\widetilde{f}$. Note that $\widetilde{f}$ is a $\widetilde{\lambda}$-strongly convex function as it is the sum of $\widetilde{T}$ convex functions and a $\widetilde{\lambda}$-strongly convex function. Thus $\widetilde{x}^{*}:=\underset{x\in\mathbb{R}^{d}}{\arg\min}\widetilde{f}(x)$ exists and is unique. We will now prove some of it's properties.

\begin{lemma} \label{lemma: 4 part minimizer lemma}
\begin{enumerate}
    \item $\forall t \in [\widetilde{T}]: \widetilde{x}^{*}_{t}\geq 0$~.
    \item $\widetilde{x}^{*}_{1}\geq\widetilde{x}^{*}_{2}\geq \dots \widetilde{x}^{*}_{\widetilde{T}}$~.
    \item $\forall t\in[\widetilde{T}-1]: \widetilde{x}_{t+1}^{*}=\widetilde{x}_{t}^{*}-(\gamma-\widetilde{\lambda}{\sum}_{j=1}^{t}\widetilde{x}_{j}^{*})^{1/k}$~.
    \item ${\sum}_{t=1}^{\widetilde{T}}\widetilde{x}_{t}^{*}=\frac{\gamma}{\widetilde{\lambda}}$~.
\end{enumerate}
\end{lemma}

\begin{proof}
\begin{enumerate}
    \item 
First we will show that $\widetilde{x}_{1}^{*}>0$. It is easily verified that $\widetilde{f}(0)=0,\nabla\widetilde{f}(0)=-\gamma e_{1}\neq0$, thus $\widetilde{f}(\widetilde{x}^{*})<0$. The only negative term in the definition of $\widetilde{f}$ is $-\gamma x_{1}$, so we deduce that $\widetilde{x}_{1}^{*}>0$. Now define $\hat{x}$ as follows: $\hat{x}_{i}=|\widetilde{x}_{i}^{*}|$. In particular $\hat{x}_{1}=\tilde{x}_{1}^{*}$, thus by the reverse triangle inequality:
\begin{equation*}
\widetilde{f}(\hat{x})-\widetilde{f}(\widetilde{x}^{*})=\frac{1}{k+1}
\sum_{i=1}^{\widetilde{T}-1}
\left||\hat{x}_{i}|-|\hat{x}_{i+1}|\right|^{k+1}-|\widetilde{x}_{i}^{*}-\widetilde{x}_{i+1}^{*}|^{k+1}\leq0~.
\end{equation*}
But $\widetilde{x}^{*}$ is the unique minimizer, so $\hat{x}=x^{*}$.

\item 
Otherwise, let $i_{0}$ be the minimal index for which $\widetilde{x}_{i_{0}}^{*}<\widetilde{x}_{i_{0}+1}^{*}$. Denote $\delta := \widetilde{x}_{i_{0}+1}^{*}-\widetilde{x}_{i_{0}}^{*}>0$, and define the vector $\hat{x}$ as follows:
\begin{equation*}
\hat{x}_{i}= 
\begin{cases}
\widetilde{x}_{i}^{*} & i\leq i_{0}\\
max\{0,\widetilde{x}_{i}^{*}-\delta\} & i>i_{0}~.
\end{cases}
\end{equation*}

Notice that $\hat{x}\neq x^{*}$ because
\begin{equation*}
    \hat{x}_{i_{0}+1}=max\{0,\widetilde{x}_{i_{0}+1}^{*}-\delta\}=max\{0,\widetilde{x}_{i_{0}}^{*}\}=\widetilde{x}_{i_{0}}^{*}=\hat{x}_{i_{0}}~,
\end{equation*}
although $x_{i_{0}}^{*}\neq x_{i_{0}+1}^{*}$ by assumption. Furthermore, 
\begin{equation*}
\forall i:|\hat{x}_{i}-\hat{x}_{i+1}|^{k+1}\leq|\widetilde{x}_{i}^{*}-\widetilde{x}_{i+1}^{*}|^{k+1},\ \|\hat{x}\|^{2}\leq\|\widetilde{x}^{*}\|^{2}~,
\end{equation*}

thus $\widetilde{x}^{*}$ is not the minimizer of $\widetilde{f}$ which is a contradiction.

\item
Differentiating $\widetilde{f}$ and setting to 0, while noticing that at the minimizer we can disregard the absolute values (due to the previous sub-lemma) we obtain the relations:

\begin{equation} \label{originally 1}
\begin{cases}
(\widetilde{x}_{1}^{*}-\widetilde{x}_{2}^{*})^{k}=\gamma-\lambda\widetilde{x}_{1}^{*}\\
(\widetilde{x}_{t}^{*}-\widetilde{x}_{t+1}^{*})^{k}=(\widetilde{x}_{t-1}^{*}-\widetilde{x}_{t}^{*})^{k}-\lambda\widetilde{x}_{t}^{*}, & \forall t\in\{2,3,...,\widetilde{T}-1\}\\
(\widetilde{x}_{\widetilde{T}-1}^{*}-\widetilde{x}_{\widetilde{T}}^{*})^{k}=\widetilde{\lambda}x_{\widetilde{T}}^{*}~.
\end{cases}
\end{equation}

The recursive relation reveals that 
$\forall t\in[\widetilde{T}-1]:(\widetilde{x}_{t}^{*}-\widetilde{x}_{t+1}^{*})^{k}=\gamma-\widetilde{\lambda}\sum_{j=1}^{t}\widetilde{x}_{j}^{*}$ which by rearranging (and recalling $\widetilde{x}_{t}^{*}\geq\widetilde{x}_{t+1}^{*})$ gives the desired result.

\item
Summing \eqref{originally 1} over $t\in\{2,...,\widetilde{T}-1\}$ reveals: 
\begin{equation*}
    (\widetilde{x}_{\widetilde{T}-1}^{*}-\widetilde{x}_{\widetilde{T}}^{*})^{k}
=(\widetilde{x}_{1}^{*}-\widetilde{x}_{2}^{*})^{k}-\widetilde{\lambda}\sum_{i=2}^{\widetilde{T}-1}\widetilde{x}_{i}^{*}~.
\end{equation*}

Plugging in the two remaining relations from \eqref{originally 1} and rearranging finishes the proof.

\end{enumerate}
\end{proof}

\begin{lemma} \label{lemma: lower bound on minimizer coordinate decay}
$\forall t\in[\widetilde{T}]:\ \widetilde{x}_{t}^{*}\geq max\left\{0,\frac{\gamma^{\frac{k+1}{2k}}}{12\sqrt{\widetilde{\lambda}}}+(\frac{1}{2}-t)\gamma^{1/k}\right\}$~.
\end{lemma}

\begin{proof}
From \lemref{lemma: 4 part minimizer lemma}.3, using \lemref{lemma: 4 part minimizer lemma}.1 we get that 
$\forall t\in[\widetilde{T}-1]:x_{t+1}^{*}\geq\widetilde{x}_{t}^{*}-\gamma^{1/k}$. Inductively: 
\begin{equation} \label{originally 2}
\forall t\in[\widetilde{T}]:\widetilde{x}_{t+1}^{*}\geq\widetilde{x}_{1}^{*}-(t-1)\gamma^{1/k}~.
\end{equation}
Using \lemref{lemma: 4 part minimizer lemma}.4 and rolling up our sleeves:
\begin{equation*}
\frac{\gamma}{\widetilde{\lambda}}
={\sum}^{\widetilde{T}}_{t=1}\widetilde{x}_{t}^{*}
\geq\sum^{\widetilde{T}}_{t=1}max\{0,\widetilde{x}_{1}^{*}-(t-1)\gamma^{1/k}\}
=\sum^{\lfloor\frac{\widetilde{x}_{1}^{*}}{\gamma^{1/k}}+1\rfloor}_{t=1}\left(\widetilde{x}_{1}^{*}-(t-1)\gamma^{1/k}\right)
\end{equation*}

\begin{equation*}
=\lfloor\frac{\widetilde{x}_{1}^{*}}{(\gamma)^{1/k}}+1\rfloor\widetilde{x}_{1}^{*}-\gamma^{1/k}\left(\frac{(\lfloor\frac{\widetilde{x}_{1}^{*}}{\gamma^{1/k}}+1\rfloor-1)\lfloor\frac{\widetilde{x}_{1}^{*}}{\gamma^{1/k}}+1\rfloor}{2}\right)
\geq\frac{(\widetilde{x}_{1}^{*})^{2}}{\gamma^{1/k}}-\frac{\gamma^{1/k}}{2}\left(\frac{\widetilde{x}_{1}^{*}}{\gamma^{1/k}}\right)\left(\frac{\widetilde{x}_{1}^{*}}{\gamma^{1/k}}+1\right)~.
\end{equation*}
By rearranging we obtain: $ (\widetilde{x}_{1}^{*})^{2}-\gamma^{1/k}\cdot\widetilde{x}_{1}^{*}-\frac{2\gamma^{1+1/k}}{\widetilde{\lambda}}\leq0 $, which implies via the quadratic formula:

\begin{equation} \label{originally 3}
\widetilde{x}_{1}^{*}\leq\frac{1}{2}\left(\gamma^{1/k}+\sqrt{\gamma^{2/k}+\frac{8(\gamma)^{1+1/k}}{\widetilde{\lambda}}}\right)
\leq\gamma^{1/k}+\sqrt{\frac{2(\gamma)^{1+1/k}}{\widetilde{\lambda}}}~.
\end{equation}
On the other hand, using \lemref{lemma: 4 part minimizer lemma}.3 again, we have that if $t\in[\widetilde{T}-1]$ satisfies $\sum^{t}_{j=1}\widetilde{x}_{j}^{*}\leq\frac{2^{k}-1}{2^{k}}\frac{\gamma}{\widetilde{\lambda}}$, then: $\widetilde{x}_{t+1}^{*}\leq\widetilde{x}_{t}^{*}-(\gamma-\frac{2^{k-1}-1}{2^{k}}\gamma)^{1/k}=\widetilde{x}_{t}^{*}-\frac{\gamma^{1/k}}{2}$. Inductively:
\begin{equation} \label{originally 4}
    \forall t\in[T-1]:{\sum^{t}_{j=1}}x_{j}^{*}\leq\frac{2^{k}-1}{2^{k}}\frac{\gamma}{\lambda} 
\Longrightarrow x_{t+1}^{*}\leq x_{1}^{*}-t\frac{\gamma^{1/k}}{2} ~.
\end{equation}
Let $t_{0}$ be the minimal index such that $\sum^{t_{0}}_{j=1}\widetilde{x}_{j}^{*}>\frac{2^{k}-1}{2^{k}}\frac{\gamma}{\widetilde{\lambda}}$ (such an index exists from \lemref{lemma: 4 part minimizer lemma}.4). Combining \lemref{lemma: 4 part minimizer lemma}.2 with \eqref{originally 3} reveals:
\begin{equation*}
\frac{2^{k}-1}{2^{k}}\frac{\gamma}{\widetilde{\lambda}}
<{\sum^{t_{0}}_{j=1}}\widetilde{x}_{j}^{*}
\leq t_{0}\widetilde{x}_{1}^{*}\leq t_{0}\left(\gamma^{1/k}+\sqrt{\frac{2(\gamma)^{1+1/k}}{\widetilde{\lambda}}}\right)
\end{equation*}
\begin{equation*}
\Longrightarrow t_{0}\geq\frac{2^{k}-1}{2^{k}}\cdot\frac{\gamma}{\widetilde{\lambda}\gamma^{1/k}+\sqrt{2\widetilde{\lambda}(\gamma)^{1+1/k}}}~.
\end{equation*}
From the minimality of $t_{0}$, we know that \eqref{originally 4} applies to $t_{0}-1$, thus: $0\leq\widetilde{x}_{t_{0}}^{*}\leq\widetilde{x}_{1}^{*}-\frac{(t_{0}-1)}{2}(\gamma)^{1/k}$. We obtain:
\begin{equation*}
\widetilde{x}_{1}^{*}\geq\frac{t_{0}\gamma^{1/k}}{2}-\frac{\gamma^{1/k}}{2}
\geq\frac{2^{k}-1}{2^{k+1}}\cdot\frac{\gamma^{1+1/k}}{\widetilde{\lambda}\gamma^{1/k}+\sqrt{2\widetilde{\lambda}\gamma^{1+1/k}}}-\frac{\gamma^{1/k}}{2}~.
\end{equation*}
Finally, plugging the last inequality into \eqref{originally 2} and rearranging gives
\begin{equation*}
\widetilde{x}_{t}^{*}\geq max\left\{0,\frac{2^{k}-1}{2^{k+1}}\left(\frac{\gamma}{\widetilde{\lambda}+\sqrt{2\widetilde{\lambda}\gamma^{1-1/k}}}\right)+\left(\frac{1}{2}-t\right)\gamma^{1/k}\right\}~.
\end{equation*}
The lemma follows easily from our assumption that $\gamma\geq\widetilde{\lambda}^{\frac{k}{k-1}}$. 
\end{proof}

\begin{lemma}
There exists an index $t_{0}\leq\frac{\widetilde{T}}{2}$ such that:
$$\widetilde{x}_{t_{0}+j}^{*}\geq\widetilde{\lambda}^{\frac{1}{k-1}}(6)^{-k^{j+1}},\forall j\in\{0,1,...,\widetilde{T}-t_{0}\}~.$$
\end{lemma}

\begin{proof}
By \lemref{lemma: 4 part minimizer lemma}.4, $\forall t\in[\widetilde{T}-1]:$
\begin{gather} \label{originally 5}
\widetilde{x}_{t}^{*}
=\widetilde{x}_{t+1}^{*}+\left(\gamma-\widetilde{\lambda}\sum^{t}_{j=1}\widetilde{x}_{j}^{*}\right)^{1/k}
=\widetilde{x}_{t+1}^{*}+\left(\gamma-\widetilde{\lambda}\left(\frac{\gamma}{\widetilde{\lambda}}-\sum_{j=t+1}^{\widetilde{T}}\widetilde{x}_{j}^{*}\right)\right)^{1/k}\\
=\widetilde{x}_{t+1}^{*}+\left(\widetilde{\lambda}\sum_{j=t+1}^{\widetilde{T}}\widetilde{x}_{j}^{*}\right)^{1/k}. \nonumber
\end{gather}
In particular, due to \lemref{lemma: 4 part minimizer lemma}.1:
\begin{equation*}
\widetilde{x}_{t}^{*}
\geq\left(\widetilde{\lambda}\sum_{j=t+1}^{\widetilde{T}}\widetilde{x}_{j}^{*}\right)^{1/k}
\geq\left(\widetilde{\lambda}\widetilde{x}_{t+1}^{*}\right)^{1/k}
\Longrightarrow\forall t\in[\widetilde{T}-1]:
\widetilde{x}_{t+1}^{*}
\leq\frac{1}{\widetilde{\lambda}}\left(\widetilde{x}_{t}^{*}\right)^{k}~,
\end{equation*}
which by induction takes the form:
\begin{equation} \label{originally 6}
    \widetilde{x}_{t+j}^{*}
\leq\frac{(\widetilde{x}_{t}^{*})^{k^{i}}}{(\widetilde{\lambda})^{\frac{k^{i}-1}{k-1}}},
\forall t,j\geq0:j+t\leq\widetilde{T}~.
\end{equation}
Now let's fix an index $t\in[\widetilde{T}-1]$ for which $\widetilde{x}_{t+1}^{*}\leq\frac{\widetilde{\lambda}^{\frac{1}{k-1}}}{2}$. Notice that such an index exists due to our assumption on $\widetilde{T}$.\footnote{
$
\frac{\gamma}{\widetilde{\lambda}}
={\sum_{j=1}^{\widetilde{T}}}\widetilde{x}_{j}^{*}
\geq\widetilde{T}\widetilde{x}_{\widetilde{T}}^{*}
>\frac{2\gamma}{\widetilde{\lambda}^{\frac{k}{k-1}}}\widetilde{x}_{\widetilde{T}}^{*}
\Longrightarrow\widetilde{x}_{\widetilde{T}}^{*}
\leq\frac{\widetilde{\lambda}^{\frac{1}{k-1}}}{2}
$} For such an index, we get from \eqref{originally 6}:

\begin{gather*}
\sum_{j=t+1}^{\widetilde{T}}\widetilde{x}_{j}^{*}
=\sum_{j=0}^{\widetilde{T}-t-1}\widetilde{x}_{t+1+j}^{*}
\leq\sum_{j=0}^{\widetilde{T}-t-1}\frac{(\widetilde{x}_{t+1}^{*})^{k^{j}}}{(\widetilde{\lambda})^{\frac{k^{j}-1}{k-1}}}\\
=\sum_{j=0}^{\widetilde{T}-t-1}\widetilde{x}_{t+1}^{*}\left(\frac{\widetilde{x}_{t+1}^{*}}{(\widetilde{\lambda})^{\frac{1}{k-1}}}\right)^{k^{j}-1}
\leq\widetilde{x}_{t+1}^{*}\sum_{j=0}^{\widetilde{T}-t-1}\left(\frac{1}{2}\right)^{k^{j}-1}
\leq2\widetilde{x}_{t+1}^{*}~.
\end{gather*}
Combining the last inequality, \eqref{originally 5}, and \lemref{lemma: 4 part minimizer lemma}.2 we get that for all $t\in[\widetilde{T}-1]$ such that $\widetilde{x}_{t+1}^{*}\leq\frac{\widetilde{\lambda}^{\frac{1}{k-1}}}{2}$:
\begin{gather*}
\widetilde{x}_{t}^{*}
=\widetilde{x}_{t+1}^{*}+\left(\widetilde{\lambda}\sum_{j=t+1}^{\widetilde{T}}\widetilde{x}_{j}^{*}\right)^{1/k}
\leq\widetilde{x}_{t+1}^{*}+\left(2\widetilde{\lambda}\widetilde{x}_{t+1}^{*}\right)^{1/k}
=\left(\widetilde{x}_{t+1}^{*}\right)^{1/k}\left((\widetilde{x}_{t+1}^{*})^{\frac{k-1}{k}}+(2\widetilde{\lambda})^{1/k}\right)\\
\leq(\widetilde{x}_{t+1}^{*})^{1/k}\left(\frac{(\widetilde{\lambda})^{1/k}}{2^{\frac{k-1}{k}}}+(2\widetilde{\lambda})^{1/k}\right)
=(\widetilde{\lambda}\widetilde{x}_{t+1}^{*})^{1/k}(2^{1-1/k}+2^{1/k})
\leq3(\widetilde{\lambda}\widetilde{x}_{t+1}^{*})^{1/k}~,
\end{gather*}
where the last inequality can be easily verified for $k\in\mathbb{N}$. Overall we have:
\begin{equation} \label{originally 7}
    \forall t\in[\widetilde{T}-1]:\widetilde{x}_{t+1}^{*}\leq\frac{\widetilde{\lambda}^{\frac{1}{k-1}}}{2}
\Longrightarrow \widetilde{x}_{t+1}^{*}\geq\frac{1}{\widetilde{\lambda}}\left(\frac{\widetilde{x}_{t}^{*}}{3}\right)^{k}~.
\end{equation}
Now let's fix $t_{0}$ as the unique index for which $\widetilde{x}_{t_{0}}^{*}>\frac{\widetilde{\lambda}^{\frac{1}{k-1}}}{2}$ and $\widetilde{x}_{t_{0}+1}^{*}\leq\frac{\widetilde{\lambda}^{\frac{1}{k-1}}}{2}$. Note that $t_{0}\leq\widetilde{T}/2$ because of our extra assumptions. Indeed, if we denote by $t_{1}$ the maximal index such that $\widetilde{x}_{t_{1}}^{*}>\frac{\widetilde{\lambda}^{\frac{1}{k-1}}}{2}$ then
\begin{equation*}
\frac{\gamma}{\widetilde{\lambda}}
\geq\sum_{t=1}^{t_{1}}\widetilde{x}_{t}^{*}
>t_{1}\cdot\frac{\widetilde{\lambda}^{\frac{1}{k-1}}}{2}
\Longrightarrow t_{1}
<\frac{2\gamma}{\widetilde{\lambda}^{\frac{k}{k-1}}}
\leq\frac{\widetilde{T}}{2}~,
\end{equation*}
where the last inequality follows from our assumption on $\widetilde{T}$. On the the other hand, it can be verified using \lemref{lemma: lower bound on minimizer coordinate decay} that $\widetilde{x}_{1}^{*}\geq\frac{\widetilde{\lambda}^{\frac{1}{k-1}}}{2}$ under our assumption $\gamma\geq(12)^{\frac{2k}{k+1}}(\widetilde{\lambda})^{\frac{2k}{2k-1}} $. Now, for all $j\leq\widetilde{T}-t_{0}$ we can use \eqref{originally 7} inductively:
\begin{gather*}
\widetilde{x}_{t_{0}+j}^{*}
\geq\frac{1}{\widetilde{\lambda}}\left(\frac{\widetilde{x}_{t_{0}+j-1}^{*}}{3}\right)^{k}
\geq\frac{(\widetilde{x}_{t_{0}+j-2}^{*})^{k^{2}}}{3^{k^{2}+k}\widetilde{\lambda}^{k+1}}
\geq...
\geq\frac{(\widetilde{x}_{t_{0}}^{*})^{k^{j}}}{3^{k+k^{2}+...+k^{j}}\widetilde{\lambda}^{1+k+...+k^{j-1}}}\\
\geq\left(\frac{\widetilde{\lambda}^{\frac{1}{k-1}}}{2}\right)^{k^{j}}\frac{1}{3^{\frac{k^{j+1}-k}{k-1}}\widetilde{\lambda}^{\frac{k^{j}-1}{k-1}}}
=\frac{\widetilde{\lambda}^{\frac{1}{k-1}}}{3^{\frac{k^{j+1}-k}{k-1}}2^{k^{j}}}
\geq\frac{\widetilde{\lambda}^{\frac{1}{k-1}}}{3^{k^{j+1}}2^{k^{j+1}}}
=\widetilde{\lambda}^{\frac{1}{k-1}}(6)^{-k^{j+1}}~.
\end{gather*}
\end{proof}

\begin{lemma} \label{lemma: minimizer norm bound}
$\sum_{i=1}^{\widetilde{T}}(\widetilde{x}_{i}^{*})^{2}
\leq\frac{3\gamma^{\frac{3k+1}{2k}}}{\widetilde{\lambda}^{\frac{3}{2}}}$~.
\end{lemma}

\begin{proof}
First we use the simple inequality:
\begin{equation*}
\sum_{i=1}^{\widetilde{T}}\widetilde{x}_{i}^{2}
\leq(\max_{i\in[\widetilde{T}]}|\widetilde{x}_{i}^{*}|)\sum_{i=1}^{\widetilde{T}}|\widetilde{x}_{i}^{*}|
=\widetilde{x}_{1}^{*}\frac{\gamma}{\widetilde{\lambda}}~,
\end{equation*}
where the last equality follows from \lemref{lemma: 4 part minimizer lemma}. Furthermore, we can upper bound the right hand side using \eqref{originally 3}:
\begin{equation*}
\widetilde{x}_{1}^{*}\frac{\gamma}{\widetilde{\lambda}}
\leq\left(\gamma^{1/k}+\sqrt{\frac{2\gamma^{\frac{k+1}{k}}}{\widetilde{\lambda}}}\right)\frac{\gamma}{\widetilde{\lambda}}
=\left(1+\sqrt{\frac{2\gamma^{\frac{k-1}{k}}}{\widetilde{\lambda}}}\right)\frac{\gamma^{\frac{k+1}{k}}}{\widetilde{\lambda}}~.
\end{equation*}
Due to our assumption $\gamma\geq\widetilde{\lambda}^{\frac{k}{k-1}}\Longrightarrow1\leq\sqrt{\frac{\gamma^{\frac{k-1}{k}}}{\widetilde{\lambda}}}$ so we can further upper bound the last term by
$$
\left(\sqrt{\frac{\gamma^{\frac{k-1}{k}}}{\widetilde{\lambda}}}+\sqrt{\frac{2(\gamma)^{\frac{k-1}{k}}}{\widetilde{\lambda}}}\right)\frac{\gamma^{\frac{k+1}{k}}}{\widetilde{\lambda}}
\leq(1+\sqrt{2})\frac{\gamma^{\frac{3k+1}{2k}}}{\widetilde{\lambda}^{\frac{3}{2}}}~,
$$
which finishes the proof since $1+\sqrt{2}<3$.
\end{proof}

\subsection{Rotated and rescaled function} \label{lower proof 2}

We will now shift our focus to a “rotated and rescaled” version of $\widetilde{f}$. Formally, let $v_{1},...,v_{\widetilde{T}}\in\mathbb{R}^{d}$ be orthogonal unit vectors, and define:
\begin{gather*}
f_{\gamma,v_{1},...,v_{\widetilde{T}}}(x_{1},...,x_{d})=\frac{\mu_{k}}{k!2^{\frac{k+3}{2}}}\left(\overset{\widetilde{T}-1}{\underset{i=1}{\sum}}g\left(\langle v_{i}-v_{i+1},x\rangle\right)-\gamma\langle v_{1},x\rangle\right)+\frac{\lambda}{2}\|x\|^{2}~,\\
g(x)=\frac{1}{k+1}|x|^{k+1}~.
\end{gather*}
For the sake of notational simplicity, we will assume $\gamma,v_{1},...,v_{\widetilde{T}}$ are somehow fixed (we will describe later on how to choose them) and denote the function by $f$ (abbreviating the parameters' subscript).
Furthermore, we will assume the relation $\lambda=\frac{\mu_{k}\widetilde{\lambda}}{k!2^{\frac{k+3}{2}}}$ (where $\widetilde{\lambda}$ is needed in order to define $\widetilde{f}$ from the previous section). Denote $x^{*}=\underset{x\in\mathbb{R}^{d}}{\arg\min}f(x)$. In order to derive properties of $x^{*}$ (analogous to subsection \ref{subsec: lower proof 1}) we will connect between $x^{*}$ and $\tilde{x}^{*}$ (from subsection \ref{subsec: lower proof 1}) through the following lemma.

\begin{lemma} \label{lemma: rotation min 1}
$\forall i\in[\widetilde{T}]:\ \langle v_{i},x^{*}\rangle =\widetilde{x}_{i}^{*}$~.
\end{lemma}

\begin{proof}
Since the global minimizer $\widetilde{x}^{*}$ is invariant under multiplication of the function $\widetilde{f}$ by a positive constant, we deduce that $\widetilde{x}^{*}$ also minimizes:
\begin{gather*}
\frac{\mu_{k}}{k!2^{\frac{k+3}{2}}}\widetilde{f}(x)=\frac{\mu_{k}}{k!2^{\frac{k+3}{2}}}\left(\overset{\widetilde{T}-1}{\underset{i=1}{\sum}}g(x_{i}-x_{i+1})-\gamma x_{1}+\frac{\widetilde{\lambda}}{2}\|x\|^{2}\right)\\
=\frac{\mu_{k}}{k!2^{\frac{k+3}{2}}}\left(\overset{\widetilde{T}-1}{\underset{i=1}{\sum}}g(x_{i}-x_{i+1})-\gamma x_{1}\right)+\frac{\lambda}{2}\|x\|^{2}~.
\end{gather*}
Now recall that $f(x)=\frac{\mu_{k}}{k!2^{\frac{k+3}{2}}}\left(\overset{\widetilde{T}-1}{\underset{i=1}{\sum}}g(\langle v_{i}-v_{i+1},x\rangle)-\gamma\langle v_{1},x\rangle\right)+\frac{\lambda}{2}\|x\|^{2}$ and that $v_{1},...,v_{\widetilde{T}}$ are orthogonal unit vectors. We deduce that $f(x)=\widetilde{C}\widetilde{f}(Vx)$ for some positive constant $\widetilde{C}$ and for any orthogonal $V$ such that it's first $\widetilde{T}$ columns are $v_{1},...,v_{\widetilde{T}}$. Therefore $x^{*}$ satisfies $Vx^{*}=(\langle v_{1},x^{*}\rangle,\langle v_{2},x^{*}\rangle,...)=\widetilde{x}^{*}$.
\end{proof}

\begin{lemma} \label{lemma: rotation min 2}
$\|x^*\|^{2}=\sum_{i=1}^{\widetilde{T}}\langle v_{i},x^{*}\rangle^{2}$~.
\end{lemma}

\begin{proof}
The proof follows immediately by combining the representer theorem and the Pythagorean theorem.
\end{proof}
Combining the previous lemmas with subsection \ref{subsec: lower proof 1} results in an immediate consequence:

\begin{proposition} \label{minimizer proposition}
Suppose that 
\begin{equation} \label{gamma bound}
\gamma>\max\left\{\left(\frac{k!2^{\frac{k+3}{2}}\lambda}{\mu_{k}}\right)^{\frac{k}{k-1}},12^{\frac{2k}{k+1}}\left(\frac{k!2^{\frac{k+3}{2}}\lambda}{\mu_{k}}\right)^{\frac{2k}{2k-1}}\right\},\  \widetilde{T}\geq\frac{4\gamma}{\left(\frac{k!2^{\frac{k+3}{2}}\lambda}{\mu_{k}}\right)^{\frac{k}{k-1}}}~.
\end{equation}
Then $f$ has a unique minimizer $x^{*}$ which satisfies:
\begin{enumerate}
    \item
    $\forall t\in[\widetilde{T}]:\langle v_{t},x^{*}\rangle\geq max\left\{0,\frac{\gamma^{\frac{k+1}{2k}}}{12\sqrt{\frac{k!2^{\frac{k+3}{2}}\lambda}{\mu_{k}}}}+(\frac{1}{2}-t)\gamma^{1/k}\right\}$~.
    
    \item
    There exists an index $t_{0}\leq\frac{\widetilde{T}}{2}$ such that for all $j\in[\widetilde{T}-t_{0}]$: 
    \begin{equation*}
       \langle v_{t_{0}+j},x^{*}\rangle\geq\left(\frac{k!2^{\frac{k+3}{2}}\lambda}{\mu_{k}}\right)^{\frac{1}{k-1}}(6)^{-k^{j+1}} ~.
    \end{equation*}
    
    \item 
    \begin{equation} \label{D bound}
        \|x^{*}\|^{2}\leq\frac{3\gamma^{\frac{3k+1}{2k}}}{\left(\frac{k!2^{\frac{k+3}{2}}\lambda}{\mu_{k}}\right)^{\frac{3}{2}}}~.
    \end{equation}
    \end{enumerate}
\end{proposition}

\subsection{Oracle complexity} \label{subsec: lower oracle complexity}
We will now derive an oracle complexity lower bound using the function we introduced in the previous section, with a particular choice of $v_{1},...,v_{\widetilde{T}}$ which will depend on the optimization algorithm. Both our function and our lower bound will still depend on the choice of $\gamma$, which we will pick in the next section. We start with a simple yet crucial lemma:

\begin{lemma} \label{lemma: orthogonal information}
If $x\in\mathbb{R}^{d}$ is orthogonal to $v_{t},v_{t+1},...,v_{\widetilde{T}}$, then $f(x),\nabla f(x),...,\nabla^{k}f(x)$ do not depend on $v_{t+1},v_{t+2},...,v_{\widetilde{T}}$.
\end{lemma}

\begin{proof}
The proof follows from looking at $f(x)-\frac{\lambda}{2}\|x\|^{2}$, and applying the chain rule while noticing that $\forall i\in[k]:g^{(i)}(0)=0$.
\end{proof}

This observation allows us to provide the following oracle complexity lower bound:

\begin{proposition} \label{proposition: T lower bound}
Assume that
\begin{equation} \label{epsilon bound}
    \epsilon<\min\left\{\frac{2^{\frac{4}{k-1}}(k!)^{\frac{2}{k-1}}(\lambda)^{\frac{k+1}{k-1}}}{(\mu_{k})^{\frac{2}{k-1}}},\frac{\lambda\gamma^{\frac{2}{k}}}{8}\right\}~.
\end{equation}
Assuming \eqref{gamma bound}, it is possible to choose the vectors $v_{1},...,v_{\widetilde{T}}$ (in the function $f$) such that the number of iterations $T$ required to have $f(x^{T})-f(x^{*})\leq\epsilon$ is at least
\begin{equation} \label{T bound}
    \max\left\{\log_{k}\log_{6}\left(\frac{2^{\frac{2}{k-1}}(k!)^{\frac{1}{k-1}}(\lambda)^{\frac{k+1}{2(k-1)}}}{(\mu_{k})^{\frac{1}{k-1}}\sqrt{\epsilon}}\right)-1,\frac{\gamma^{\frac{k-1}{2k}}}{12\left(\frac{k!2^{\frac{k+3}{2}}\lambda}{\mu_{k}}\right)^{\frac{1}{2}}}\right\}~.
\end{equation}
\end{proposition}

\begin{proof}
Fix the number of iterations $T\leq\widetilde{T}$. Given some algorithm, we will describe a method of picking $v_{1},...,v_{\widetilde{T}}$ in such a way that will result in a the desired lower bound.

\begin{itemize}
    \item
    First, compute $x^{1}$ (which is the deterministic algorithm's first output, thus chosen without any oracle calls are even made).
    
    \item
    Pick $v_{1}$ to be some unit vector orthogonal to $x^{1}$. Assuming $v_{2},...,v_{\widetilde{T}}$ will also be orthogonal to $x^{1}$ (which we will soon ensure by construction), we have by \lemref{lemma: orthogonal information} that $f(x^{1}),...,\nabla^{k}f(x^{1})$ do not depend on $v_{2},...,v_{\widetilde{T}}$, thus depend only on $v_{1}$ which is already fixed. Since the algorithm is deterministic, this fixes the next point $x^{2}$.
    
    \item
    Repeat the same process for $t=2,3,...,T-1$: Compute $x^{t}$, and pick $v_{t}$ to be some unit vector orthogonal to $x^{1},...,x^{t}$ and also to $v_{1},...,v_{t-1}$ (recall the dimension is assumed to be large enough). Similar to the previous step, \lemref{lemma: orthogonal information} insures the next point $x^{t+1}$ is fixed.
    
    \item
     At the end of the process, pick $v_{T},...,v_{\widetilde{T}}$ to be some unit vectors which are orthogonal to all previously chosen $v$'s as well as $x^{1},...,x^{T}$ (which is possible due to the large dimension).

\end{itemize}

We get by this construction (where $t_{0}$ is defined in proposition \ref{minimizer proposition}):
\begin{equation*}
\|x^{T}-x^{*}\|^{2}
\geq\sum_{i=1}^{\widetilde{T}}\langle v_{i},x_{T}-x^{*}\rangle^{2}
\geq\langle v_{t_{0}+T},x_{T}-x^{*}\rangle^{2}
=\langle v_{t_{0}+T},x^{*}\rangle^{2}~,
\end{equation*}
which can be lower bounded using proposition \ref{minimizer proposition} by $\left(\frac{k!2^{\frac{k+3}{2}}\lambda}{\mu_{k}}\right)^{\frac{2}{k-1}}(6)^{-2k^{T+1}}$. Using the strong convexity of $f$, we get:
\begin{equation*}
f(x^{T})-f(x^{*})
\geq\frac{\lambda}{2}\|x^{T}-x^{*}\|^{2}
\geq\frac{\lambda}{2}
\left(\frac{k!2^{\frac{k+3}{2}}\lambda}{\mu_{k}}\right)^{\frac{2}{k-1}}(6)^{-2k^{T+1}}~.
\end{equation*}
In order to make the last expression smaller than $\epsilon$, $T$ must satisfy:
\begin{equation*}
\frac{\lambda}{2}\left(\frac{k!2^{\frac{k+3}{2}}\lambda}{\mu_{k}}\right)^{\frac{2}{k-1}}(6)^{-2k^{T+1}}\leq\epsilon
\iff
k^{T+1}\geq\log_{6}\left(\frac{2^{\frac{2}{k-1}}(k!)^{\frac{1}{k-1}}(\lambda)^{\frac{k+1}{2(k-1)}}}{(\mu_{k})^{\frac{1}{k-1}}\sqrt{\epsilon}}\right)~.
\end{equation*}
Since we assume $\epsilon<\frac{2^{\frac{4}{k-1}}(k!)^{\frac{2}{k-1}}(\lambda)^{\frac{k+1}{k-1}}}{(\mu_{k})^{\frac{2}{k-1}}}$ the expression inside the logarithm is greater than $1$, so we get
\begin{equation*}
T\geq\log_{k}\log_{6}\left(\frac{2^{\frac{2}{k-1}}(k!)^{\frac{1}{k-1}}(\lambda)^{\frac{k+1}{2(k-1)}}}{(\mu_{k})^{\frac{1}{k-1}}\sqrt{\epsilon}}\right)-1~,
\end{equation*}
which is the first part of the lower bound. For the second part of the lower bound, notice that:
\begin{gather*}
\|x^{T}-x^{*}\|^{2}
\geq\sum_{i=1}^{T}\langle v_{i},x_{T}-x^{*}\rangle^{2}
\geq\langle v_{T},x_{T}-x^{*}\rangle^{2}=\langle v_{T},x^{*}\rangle^{2}\\
\Longrightarrow f(x^{T})-f(x^{*})
\geq\frac{\lambda}{2}\|x^{T}-x^{*}\|^{2}
\geq\frac{\lambda}{2}\langle v_{T},x^{*}\rangle^{2}\\
\overset{prop. \ref{minimizer proposition}}{\geq}\frac{\lambda}{2}\left(\frac{\gamma^{\frac{k+1}{2k}}}{12\sqrt{\frac{k!2^{\frac{k+3}{2}}\lambda}{\mu_{k}}}}+(\frac{1}{2}-T)\gamma^{1/k}\right)^2~.
\end{gather*}
In order to make the right hand side smaller than $\epsilon$, $T$ must satisfy:
\begin{equation*}
\frac{\lambda}{2}\left(\frac{\gamma^{\frac{k+1}{2k}}}{12\sqrt{\frac{k!2^{\frac{k+3}{2}}\lambda}{\mu_{k}}}}+(\frac{1}{2}-T)\gamma^{1/k}\right)^2\leq\epsilon
\iff
T\geq\frac{\gamma^{\frac{k-1}{2k}}}{12\left(\frac{k!2^{\frac{k+3}{2}}\lambda}{\mu_{k}}\right)^{\frac{1}{2}}}-\frac{\sqrt{2\epsilon}}{\gamma^{\frac{1}{k}}\sqrt{\lambda}}+\frac{1}{2}~,
\end{equation*}
and since we assumed $\epsilon<\frac{\lambda\gamma^{\frac{2}{k}}}{8}$, we have that $-\frac{\sqrt{2\epsilon}}{\gamma^{\frac{1}{k}}\sqrt{\lambda}}+\frac{1}{2}>0$, finishing the proof of the proposition.
\end{proof}

\subsection{Setting $\gamma$ and wrapping up} \label{lower proof 4}

We can assume without loss of generality that the algorithm initializes at $x^{1}=0$ (otherwise, simply replace $f(x)$ by $f(x-x^{1})$). Thus, the theorem requires that our function's minimizer $x^{*}$ will satisfy $\|x^{*}\|\leq D$. We will soon establish this using previous sections, but first we'll show that $f$ indeed has the structural properties we desire regardless of the choice of $\gamma$.

\begin{lemma}
$f$ is $\lambda$-strongly convex, $k$ times differentiable, with $\mu_{k}$-Lipschitz k-th derivative.
\end{lemma}

\begin{proof}
$f$ is $\lambda$-strongly convex as a sum of convex functions and the $\lambda$-strongly convex $\frac{\lambda}{2}\|x\|^{2}$, and is easily verified to be $k$ times smooth. It remains to show the desired Lipschitz property. Since Lipschitzness is invariant under orthogonal transformations, we may assume without loss of generality that
\begin{equation*}
v_{i}=\mathbf{e}_{i}\Longrightarrow
f(x)=\frac{\mu_{k}}{k!2^{\frac{k+3}{2}}}\left(\overset{\widetilde{T}-1}{\underset{i=1}{\sum}}g(\langle\mathbf{e}_{i}-\mathbf{e}_{i+1},x\rangle)-\gamma x_{1}\right)+\frac{\lambda}{2}\|x\|^{2}~.
\end{equation*}
Denote $\mathbf{r}_{i}=\mathbf{e}_{i}-\mathbf{e}_{i+1}$, and assume for simplicity that $k>2$ (the case $k=2$ can be verified in a similar manner). The chain rule implies
\begin{equation*}
\nabla^{(k)}f(x)=\frac{\mu_{k}}{k!2^{\frac{k+3}{2}}}\left(\overset{\widetilde{T}-1}{\underset{i=1}{\sum}}g^{(k)}(\langle\mathbf{r}_{i},x\rangle)\mathbf{r}_{i}^{\otimes k}\right)~,
\end{equation*}
where $v^{\otimes k}=\overset{k}{\overbrace{v\otimes v\otimes...\otimes v}}$. Since $g^{(k)}(x)=k!|x|$ we have
\begin{align*}
\|\nabla^{(k)}f(x)-\nabla^{(k)}f(\bar{x})\|
&=\frac{\mu_{k}k!}{k!2^{\frac{k+3}{2}}}\left\|\overset{\widetilde{T}-1}{\underset{i=1}{\sum}}(|\langle\mathbf{r}_{i},x\rangle|-|\langle\mathbf{r}_{i},\bar{x}\rangle|)\mathbf{r}_{i}^{\otimes k}\right\| \\
&\leq\frac{\mu_{k}}{2^{\frac{k+3}{2}}}\left\|\overset{\widetilde{T}-1}{\underset{i=1}{\sum}}(|\langle\mathbf{r}_{i},x-\bar{x}\rangle|)\mathbf{r}_{i}^{\otimes k}\right\|~,
\end{align*}
where the last inequality is due to the reverse triangle inequality, and the linearity of the dot product. Recall that the operator norm of a tensor is defined as
$\|T\|=\underset{\|z\|=1}{\max}\left|\underset{i_{1},..,i_{k}}{\sum}T_{i_{1},...,i_{k}}z_{i_{1}}...z_{i_{k}}\right|$, so

\begin{gather*}
\left\|\overset{\widetilde{T}-1}{\underset{i=1}{\sum}}(|\langle\mathbf{r}_{i},x-\bar{x}\rangle|)\mathbf{r}_{i}^{\otimes k}\right\|
=\underset{\|z\|=1}{\max}\left|\underset{i_{1},..,i_{k}}{\sum}\overset{\widetilde{T}-1}{\underset{i=1}{\sum}}|\langle\mathbf{r}_{i},x-\bar{x}\rangle|r_{i,i_{1}}...r_{i,i_{k}}z_{i_{1}}...z_{i_{k}}\right|\\
=\underset{\|z\|=1}{\max}\left|\overset{\widetilde{T}-1}{\underset{i=1}{\sum}}|\langle\mathbf{r}_{i},x-\bar{x}\rangle|\underset{i_{1}}{\sum}r_{i,i_{1}}z_{i_{1}}\cdots\underset{k}{\sum}r_{i,i_{k}}z_{i_{k}}\right|
=\underset{\|z\|=1}{\max}\left|\overset{\widetilde{T}-1}{\underset{i=1}{\sum}}|\langle\mathbf{r}_{i},x-\bar{x}\rangle|\langle\mathbf{r}_{i},z\rangle^{k}\right|\\
\leq2^{\frac{k}{2}}\underset{\|z\|=1}{\max}\overset{\widetilde{T}-1}{\underset{i=1}{\sum}}|\langle\mathbf{r}_{i},x-\bar{x}\rangle|\cdot\left|\langle\frac{\mathbf{r}_{i}}{\sqrt{2}},z\rangle\right|^{k}
=2^{\frac{k+1}{2}}\underset{\|z\|=1}{\max}\overset{\widetilde{T}-1}{\underset{i=1}{\sum}}\left|\langle\frac{\mathbf{r}_{i}}{\sqrt{2}},x-\bar{x}\rangle\right|\cdot\left|\langle\frac{\mathbf{r}_{i}}{\sqrt{2}},z\rangle\right|^{k}\\
\leq2^{\frac{k+1}{2}}\|x-\bar{x}\|\underset{\|z\|=1}{\max}\overset{\widetilde{T}-1}{\underset{i=1}{\sum}}\left|\langle\frac{\mathbf{r}_{i}}{\sqrt{2}},z\rangle\right|^{2}
\leq2^{\frac{k-1}{2}}\|x-\bar{x}\|\underset{\|z\|=1}{\max}\overset{\widetilde{T}-1}{\underset{i=1}{\sum}}(z_{i}-z_{i+1})^{2}\\
\leq2^{\frac{k-1}{2}}\|x-\bar{x}\|\underset{\|z\|=1}{\max}\left(2\cdot\overset{\widetilde{T}-1}{\underset{i=1}{\sum}}z_{i}^{2}+2\cdot\overset{\widetilde{T}}{\underset{i=2}{\sum}}z_{i}^{2}\right)\leq2^{\frac{k+3}{2}}\|x-\bar{x}\|~.\\
\end{gather*}
So overall we got
$$
\|\nabla^{(k)}f(x)-\nabla^{(k)}f(\bar{x})\|\leq\frac{\mu_{k}}{2^{\frac{k+3}{2}}}2^{\frac{k+3}{2}}\|x-\bar{x}\|=\mu_{k}\|x-\bar{x}\|~.
$$
\end{proof}

Now all that remains is to pick $\gamma$ in order to finish the proof of the Theorem. Our goal is to maximize the oracle complexity lower bound \eqref{T bound} under the constraints stated in \eqref{gamma bound}, \eqref{epsilon bound} and $\|x^*\|\leq D$. Note that in order to satisfy the latter, it is enough to bound the right hand side of \eqref{D bound} by $D^2$. An elementary computation shows one should pick $\gamma=\left(\frac{2^{\frac{3k+9}{2}}\left(k!\right)^{3}\lambda^{3}D^{4}}{9\mu_{k}^{3}}\right)^{\frac{k}{3k+1}}$, which results in the main theorem.

\section{Proof of Theorem \ref{upper bound main result}}
Our proof is constructive - we will describe such an algorithm. Our algorithm is based on a certain combination of accelerated Taylor descent (ATD) \cite{bubeck2019near} with cubic regularized Newton (CRN) \cite{nesterov2008accelerating},
while analyzing them under the highly smooth strongly convex regime.
More specifically, we suggest an algorithm which is composed of several different phases. We build the theoretical background during subsections \ref{upper proof 1}, \ref{upper proof 2}, and combine them into \ref{upper proof 3}.

\subsection{Highly smooth strongly convex ATD} \label{upper proof 1}
Assume $f$ is as stated in the Theorem. We start by proving a basic property of high order smoothness.

\begin{lemma} \label{high order smoothness lemma}
For all $x\in\mathbb{R}^d$ such that $\|x-x^*\|\leq D$:
\begin{equation*}
f(x)-f(x^{*})\leq\frac{M}{4}\cdot D^{k-1}+\frac{\mu_{k}}{2}\cdot D^{k+1}~.
\end{equation*}
\end{lemma}

\begin{proof}
The proof is based on recursively applying Taylor's theorem.
\begin{gather*}
f(x)-f(x^{*})=\nabla f(x^{*})\cdot(x-x^{*})+\frac{1}{2}\nabla^{2}f(x^{(2)})\cdot\left(x-x^{*}\right)^{\otimes2}\\
\leq\frac{1}{2}\|\nabla^{2}f(x^{(2)})\|\cdot\|x-x^{*}\|^{2}~,
\end{gather*}
where $x^{(2)}$ lies on the segment between $x^{*},x$. Furthermore
\begin{gather*}
\|\nabla^{2}f(x^{(2)})-\nabla^{2}f(x^{*})\|\leq\left\|\nabla^{3}f(x^{(3)})\cdot\left(x^{(2)}-x^{*}\right)\right\| \\
\leq\|\nabla^{3}f(x^{(3)})\|\cdot\|x^{(2)}-x^{*}\|\leq\|\nabla^{3}f(x^{(3)})\|\cdot\|x-x^{*}\|~.
\end{gather*}
Similarly for all $2\leq j<k$:
\begin{gather*}
\|\nabla^{j}f(x^{(j)})-\nabla^{j}f(x^{*})\|\leq\left\|\nabla^{j+1}f(x^{(j+1)})\cdot\left(x^{(j)}-x^{*}\right)\right\| \\
\leq\|\nabla^{j+1}f(x^{(j+1)})\|\cdot\|x-x^{*}\|~,
\end{gather*}
with each $x^{(j+1)}$ lying on the line segment between $x^{(j)},x^{*}$. Eventually,
\begin{equation*}
\|\nabla^{k}f(x^{(k)})-\nabla^{k}f(x^{*})\|\leq\mu_{k}\|x^{(k)}-x^{*}\|\leq\mu_{k}\|x-x^{*}\|~.
\end{equation*}
Recursively combining the triangle inequality with the previous inequalities gives
\begin{gather*}
\|\nabla^{2}f(x^{(2)})\|\leq\|\nabla^{2}f(x^{*})\|+\|\nabla^{2}f(x^{(2)})-\nabla^{2}f(x^{*})\| \\
\leq\|\nabla^{2}f(x^{*})\|+\|\nabla^{3}f(x^{(3)})\|\cdot\|x-x^{*}\| \\
\leq\|\nabla^{2}f(x^{*})\|+\left(\|\nabla^{3}f(x^{*})\|+\|\nabla^{3}f(x^{(3)})-\nabla^{3}f(x^{*})\|\right)\cdot\|x-x^{*}\| \\
\ldots\leq\sum_{j=2}^{k-1}\|\nabla^{j}f(x^{*})\|\cdot\|x-x^{*}\|^{j-2}+\|\nabla^{k}f(x^{(k)})-\nabla^{k}f(x^{*})\|\cdot\|x-x^{*}\|^{k-2} \\
\leq\sum_{j=2}^{k-1}\|\nabla^{j}f(x^{*})\|\cdot\|x-x^{*}\|^{j-2}+\mu_{k}\cdot\|x-x^{*}\|^{k-1} \\
\leq M\cdot\sum_{j=2}^{k-1} D^{j-2}+\mu_{k}\cdot D^{k-1}~.
\end{gather*}
So overall,
\begin{gather*}
f(x)-f(x^{*})\leq\frac{1}{2}\left(M\cdot\sum_{j=2}^{k-1} D^{j-2}+\mu_{k}\cdot D^{k-1}\right)\cdot D^{2} \\
=\frac{1}{2}\left(M\cdot\sum_{j=2}^{k-1} D^{j}+\mu_{k}\cdot D^{k+1}\right)
\leq\frac{M}{2}\cdot\sum_{j=2}^{k-1}D^{j}+\frac{\mu_{k}}{2}\cdot D^{k+1} \\
=\frac{M}{2}\cdot\frac{D^{k}-D^{2}}{D-1}+\frac{\mu_{k}}{2}\cdot D^{k+1}
\overset{D\geq2}{\leq}\frac{M}{2}\cdot\frac{D^{k}-D^{2}}{D/2}+\frac{\mu_{k}}{2}\cdot D^{k+1} \\
=\frac{M}{4}\cdot D^{k-1}+\frac{\mu_{k}}{2}\cdot D^{k+1}~.
\end{gather*}
\end{proof}
We are now ready to describe how to utilize ATD in the presence of high smoothness and strong convexity. ATD satisfies
\begin{equation*}
f(x^{t})-f(x^{*})\leq\frac{c_{k}\mu_{k}\|x^{1}-x^{*}\|^{k+1}}{t^{\frac{3k+1}{2}}}~.
\end{equation*}
Strong convexity implies
\begin{equation*}
\forall x:\ \frac{\lambda}{2}\|x-x^{*}\|^{2}\leq f(x)-f(x^{*})~,
\end{equation*}
and by assumption $\|x^{1}-x^{*}\|\leq D$ so we can get
\begin{equation*}
 f(x^{t})-f(x^{*})\leq\frac{c_{k}\mu_{k}\|x^{1}-x^{*}\|^{k+1}}{t^{\frac{3k+1}{2}}}\leq\frac{2c_{k}\mu_{k}D^{k-1}}{\lambda t^{\frac{3k+1}{2}}}\left(f(x^{1})-f(x^{*})\right)~.
\end{equation*}
If we set $\tau$ such that
\begin{equation*}
\frac{2c_{k}\mu_{k}D^{k-1}}{\lambda\tau^{\frac{3k+1}{2}}}=\frac{1}{2}\iff\tau=\left(\frac{4c_{k}\mu_{k}D^{k-1}}{\lambda}\right)^{\frac{2}{3k+1}}~,
\end{equation*}
we get after $\tau$ iterations 
\begin{equation*}
f(x_{\tau})-f(x^{*})\leq\frac{1}{2}\left(f(x_{1})-f(x^{*})\right)~.
\end{equation*}
Now notice that although accelerated schemes are not monotone in general, we have not gotten further away than our original starting point. In order to see this notice that our choice of $\tau$ satisfies
\begin{equation*}
    \tau\geq\left(\frac{2c_{k}\mu_{k}D^{k-1}}{\lambda}\right)^{\frac{2}{3k+1}}
    \implies
    \frac{c_k\mu_k D^{k+1}}{\tau^{\frac{3k+1}{2}}}\leq\frac{\lambda D^2}{2}~.
\end{equation*}
Thus we can utilize strong convexity and ATD's sub-optimality guarantee and get
\begin{equation*}
    \|x^{\tau}-x^*\| 
    \leq \sqrt{\frac{2}{\lambda}\left(f(x^{\tau})-f(x^*)\right)}
    \leq \sqrt{\frac{2}{\lambda}\left(\frac{c_{k}\mu_{k}\|x^{1}-x^{*}\|^{k+1}}{{\tau}^{\frac{3k+1}{2}}}\right)}
    \leq \sqrt{\frac{2}{\lambda}\cdot\frac{\lambda D^2}{2}}=D~.
\end{equation*}
Using this crucial observation, we can now initialize ATD with $x^{1}=x^{\tau}$. Iterating this process eventually gives
\begin{equation*}
f(x^{T})-f(x^{*})\leq\frac{f(x^{1})-f(x^{*})}{2^{T/\tau}}~.
\end{equation*}
In order to make the right hand side smaller than some $\epsilon_0$, a simple rearrangement reveals it is enough to set
\begin{equation*}
T\geq\left(\frac{4c_{k}\mu_{k}D^{k-1}}{\lambda}\right)^{\frac{2}{3k+1}}\cdot\log_{2}\left(\frac{f(x^{1})-f(x^{*})}{\epsilon}\right)~.
\end{equation*}
Using Lemma \ref{high order smoothness lemma}, we conclude that it is enough to set
\begin{equation} \label{ATD phase T equation}
T=\left(\frac{4c_{k}\mu_{k}D^{k-1}}{\lambda}\right)^{\frac{2}{3k+1}}\cdot\log_{2}\left(\frac{MD^{k-1}+2\mu_{k}\cdot D^{k+1}}{4\epsilon}\right)~.
\end{equation}

\subsection{Cubic regularized Newton} \label{upper proof 2}
If we were to assume that $\nabla^{2}f$ is $\mu_2$-Lipschitz, then provided that we start inside the region $$\left\{ f(x)-f(x^{*})\leq\frac{\lambda^{3}}{2\mu_{2}^{2}}\right\} $$ CRN guarantees $\epsilon$ sub-optimality after $ \mathcal{O}\left(\log\log\left(\frac{\lambda^{3}}{\mu_{2}^{2}\epsilon}\right)\right)$ steps (under our assumption on $\epsilon$). Naively, we would like to substitute $\epsilon=\frac{\lambda^{3}}{2\mu_{2}^{2}}$ into the iteration bound from the previous subsection and introduce a second phase in which the algorithm follows CRN. The obstacle is that we do not want to assume that $\mu_{2}$ exists globally. Thus we define
\begin{equation*}
\widetilde{\mu}_2(r)
:=
\sup\{L>0 \mid  \forall x,y\in B(x^*,r): \|\nabla^2f(x)-\nabla^2f(y)\|\leq L\|x-y\|
\}~.
\end{equation*}
That is, a $\mu_2$ proxy in a ball of radius $r$ around the optimum. Notice that
\begin{equation*}
\widetilde{\mu}_2(r)
\leq
\sup\left\{\|\nabla^{3}f(x)\| \mid  x\in\overline{B(x^*,r)}\right\}~.
\end{equation*}
The right hand side can be bounded exactly the same way as in the proof of Lemma \ref{high order smoothness lemma}, resulting in
\begin{equation*}
M\cdot\sum_{j=3}^{k-1}r^{j-3}+\mu_{k}\cdot r^{k-2}~,
\end{equation*}
which if $r\neq1$ equals
\begin{equation*}
M\cdot\frac{r^{k-3}-1}{r-1}+\mu_{k}\cdot r^{k-2}~.
\end{equation*}
Overall, for all $r\neq1:$
\begin{equation} \label{mu2 r equation}
\widetilde{\mu}_2(r)\leq
M\cdot\frac{r^{k-3}-1}{r-1}+\mu_{k}\cdot r^{k-2}~.
\end{equation}
In particular, this results in the following lemma:
\begin{lemma} \label{r insures mu2 bound lemma}
If $r<\min\left\{1,\left(\frac{M}{\mu_k}\right)^{\frac{1}{k-2}}\right\}$ then $\widetilde{\mu}_2(r)\leq2M$.
\end{lemma}
\begin{proof}
For such $r$ it holds that $\frac{r^{k-3}-1}{r-1}\leq1$ and also $\mu_{k}\cdot r^{k-2}<M$. Thus
\begin{equation*}
M\cdot\frac{r^{k-3}-1}{r-1}+\mu_{k}\cdot r^{k-2}\leq2M~,
\end{equation*}
which by \eqref{mu2 r equation} finishes the proof.
\end{proof}

\subsection{Combining the ingredients} \label{upper proof 3}
Our suggested algorithm starts by performing ATD for $t_1$ steps, where $t_1$ is the minimal index for which $\|x^{t_1}-x^*\|< \min\left\{1,\left(\frac{M}{\mu_k}\right)^{\frac{1}{k-2}}\right\}$. We will now calculate $t_1$. Using strong convexity, we know that it is enough to ensure
\begin{gather*}
\sqrt{\frac{2}{\lambda}\left(f(x^{t_{1}})-f(x^{*})\right)}\leq \min\left\{1,\left(\frac{M}{\mu_k}\right)^{\frac{1}{k-2}}\right\} \\
\iff
f(x^{t_{1}})-f(x^{*})<\min\left\{ \frac{\lambda}{2},\frac{\lambda}{2}\left(\frac{M}{\mu_{k}}\right)^{\frac{2}{k-2}}\right\}~.
\end{gather*}
Equation \eqref{ATD phase T equation} allows us to convert the latter into a precise statement about $t_1$. A simple calculation turns out to give
\begin{equation} \label{t_1 value equation}
    t_1=\left(\frac{4c_{k}\mu_{k}D^{k-1}}{\lambda}\right)^{\frac{2}{3k+1}}\left(\log_{2}\left(\frac{\left(MD^{k-1}+2\mu_{k}\cdot D^{k+1}\right)^{2}}{8\lambda}\cdot\left(\frac{\mu_{k}}{M}\right)^{\frac{2}{k-2}}\right)\right)~.
\end{equation}
Afterwards, we perform $t_2$ more steps of ATD, where $t_2$ is the minimal index for which
\begin{equation*}
f(x^{t_{2}})-f(x^{*})\leq\frac{\lambda^{3}}{32M^{2}}~.
\end{equation*}
Once again we use equation \eqref{ATD phase T equation}, in order to deduce
\begin{equation} \label{t_2 value equation}
    t_2=\left(\frac{4c_{k}\mu_{k}D^{k-1}}{\lambda}\right)^{\frac{2}{3k+1}}\cdot\log_{2}\left(\frac{8M^{3}D^{k-1}+16M^{2}\mu_{k}\cdot D^{k+1}}{\lambda^{3}}\right)~.
\end{equation}
At this point, the algorithm starts performing cubic regularized Newton steps. Notice that the whole point of waiting the first $t_1$ steps was to ensure Lemma \eqref{r insures mu2 bound lemma} is applicable. Furthermore, the additional $t_2$ steps ensure we are in the region of quadratic convergence for CRN, with $\mu_2=\Theta(M)$. Thus, from this point onward the number of steps required to ensure $\epsilon$ sub-optimality is
\begin{equation} \label{CRN phase convergence}
    \mathcal{O}\left(\log\log\left(\frac{\lambda^{3}}{M^{2}\epsilon}\right)\right)~.
\end{equation}
Summing \eqref{t_1 value equation}, \eqref{t_2 value equation}, \eqref{CRN phase convergence} and slightly rearranging finishes the proof.

\bibliographystyle{plainnat}
\bibliography{high_order_v2.bib}

\appendix

\section{Proof of Theorem \ref{thm: low dim lower bound}} \label{Appendix: proof of low dim lower bound}
The proof follows the same path as Section \ref{sec: proof of main lower bound}. To that end, consider the function
\begin{equation*}
    \widetilde{f}_{\gamma}(x_{1},...,x_{d})=
    \frac{1}{k+1} 
    \sum_{i=1}^{d-1}
    \left|x_{i}-x_{i+1}\right|^{k+1}-\gamma x_{1}+
    \frac{\widetilde{\lambda}}{2}\|x\|^{2}
\end{equation*}
where $\gamma>0,\ \widetilde{\lambda}=\frac{k!2^{\frac{k+3}{2}}\lambda}{\mu_k}$. Since $\widetilde{f}_\gamma$ is $\widetilde{\lambda}$-strongly convex, $\widetilde{x}^{*}:=\underset{x\in\mathbb{R}^{d}}{\arg\min}\widetilde{f}_\gamma(x)$ exists and is unique. We start by stating properties of $\widetilde{x}^*$ which are analogous to those proved in subsection \ref{subsec: lower proof 1}. The slight differences between the two settings are the result of substituting $\widetilde{T}=d$, and the lack of the assumptions on $d,\gamma$ which are made in the beginning of subsection \ref{subsec: lower proof 1}.
\begin{lemma} \label{lemma: low dim 4 part minimizer lemma}
\begin{enumerate}
    \item $\forall t \in [d]: \widetilde{x}^{*}_{t}\geq 0$~.
    \item $\widetilde{x}^{*}_{1}\geq\widetilde{x}^{*}_{2}\geq \dots \widetilde{x}^{*}_{d}$~.
    \item $\forall t\in[d-1]: \widetilde{x}_{t+1}^{*}=\widetilde{x}_{t}^{*}-(\gamma-\widetilde{\lambda}{\sum}_{j=1}^{t}\widetilde{x}_{j}^{*})^{1/k}$~.
    \item ${\sum}_{t=1}^{d}\widetilde{x}_{t}^{*}=\frac{\gamma}{\widetilde{\lambda}}$~.
\end{enumerate}
\end{lemma}
\begin{proof}
Identical to the proof of \lemref{lemma: 4 part minimizer lemma}, with $\widetilde{T}=d$.
\end{proof}
\begin{lemma} 
$\forall t\in[d]:\ \widetilde{x}_{t}^{*}\geq max\left\{0,\frac{1}{4}\cdot\frac{\gamma}{\widetilde{\lambda}+\sqrt{2\widetilde{\lambda}\gamma^{\frac{k-1}{k}}}}+\left(\frac{1}{2}-t\right)\gamma^{1/k}\right\}$~.
\end{lemma}
\begin{proof}
Identical to the proof of \lemref{lemma: lower bound on minimizer coordinate decay} up until the further assumption on $\gamma$, while noticing that $\frac{2^k-1}{2^{k+1}}\geq\frac{1}{4}$.
\end{proof}
\begin{lemma}
$\sum_{i=1}^{d}(\widetilde{x}_{i}^{*})^{2}
\leq
\left(1+\sqrt{\frac{2\gamma^{\frac{k-1}{k}}}{\widetilde{\lambda}}}\right)\frac{\gamma^{\frac{k+1}{k}}}{\widetilde{\lambda}}$~.
\end{lemma}
\begin{proof}
Identical to the proof of \lemref{lemma: minimizer norm bound}, up until the further assumption on $\gamma$. Stopping there gives the result.
\end{proof}
Let $v_1,\dots,v_d\in\mathbb{R}^d$ be orthogonal unit vectors, and define:
\begin{gather*}
f_{\gamma,v_{1},...,v_{d}}(x_{1},...,x_{d})=\frac{\mu_{k}}{k!2^{\frac{k+3}{2}}}\left(\overset{d-1}{\underset{i=1}{\sum}}g\left(\langle v_{i}-v_{i+1},x\rangle\right)-\gamma\langle v_{1},x\rangle\right)+\frac{\lambda}{2}\|x\|^{2}~,\\
g(x)=\frac{1}{k+1}|x|^{k+1}~.
\end{gather*}
Assuming $\gamma,v_{1},...,v_{d}$ are somehow fixed, we use the abbreviation $f$ and denote $x^{*}=\underset{x\in\mathbb{R}^{d}}{\arg\min}f(x)$. Note that \lemref{lemma: rotation min 1} and \lemref{lemma: rotation min 2} hold as is with $\widetilde{T}=d$. Combined with all the previous lemmas this results in the following proposition, analogous to Proposition \ref{minimizer proposition}.
\begin{proposition} \label{proposition: minimizer low dim}
Suppose that $\gamma>0$. Then $f$ has a unique minimizer $x^{*}$ which satisfies: 
\begin{enumerate}
    \item
    $\forall t\in[d]:\langle v_{t},x^{*}\rangle\geq max\left\{0,\frac{1}{4}\cdot\frac{\gamma}{\widetilde{\lambda}+\sqrt{2\widetilde{\lambda}\gamma^{\frac{k-1}{k}}}}+\left(\frac{1}{2}-t\right)\gamma^{1/k}\right\}$~.
    \item 
    $\|x^{*}\|^{2}\leq\left(1+\sqrt{\frac{2\gamma^{\frac{k-1}{k}}}{\widetilde{\lambda}}}\right)\frac{\gamma^{\frac{k+1}{k}}}{\widetilde{\lambda}}$~.
    \end{enumerate}
\end{proposition}

We continue with deriving an oracle complexity lower bound in the same way as in subsection \ref{subsec: lower oracle complexity}. Notice that \lemref{lemma: orthogonal information} holds as is, so we can perform the exact same orthogonality procedure of proposition \ref{proposition: T lower bound} for $T=d/10$ iterations (this will require a dimension which is smaller than $d$ in order to ensure the existence of the orthogonal directions). This results in
\begin{gather*}
    \|x^T-x^*\|^2
    = \sum_{i=1}^{d}\langle v_i,x^T-x^*\rangle^2
    \geq \langle v_T,x^T-x^*\rangle^2
    = \langle v_T,x^*\rangle^2
    \\
    \implies f\left(x^T\right)-f\left(x^*\right)
    \geq \frac{\lambda}{2}\|x^T-x^*\|^2
    \geq \frac{\lambda}{2}\langle v_T,x^*\rangle^2
    \\
    \geq
    \frac{\lambda}{2}\left[\frac{1}{4}\cdot\frac{\gamma}{\widetilde{\lambda}+\sqrt{2\widetilde{\lambda}\gamma^{\frac{k-1}{k}}}}+\left(\frac{1}{2}-\frac{d}{10}\right)\gamma^{1/k}\right]^2~.
\end{gather*}
Elementary rearrangements show that if 
\begin{equation} \label{epsilon,d bound low dim}
    \epsilon\leq\frac{\gamma^{2/k}\lambda}{32}~,\ \  d\leq\frac{10\gamma^{\frac{k-1}{k}}}{4\left(\widetilde{\lambda}+\sqrt{2\widetilde{\lambda}\gamma^{\frac{k-1}{k}}}\right)}+\frac{10}{4}~,
\end{equation}
then the former expression is larger than $\epsilon$ which leads to the desired $T\geq d/10=\Omega(d)$ lower bound. Since we assume $d\leq\left(\frac{D^{k-1}}{\widetilde{\lambda}}\right)^{\frac{2}{3k+1}}$, by \eqref{epsilon,d bound low dim} in order to finish the proof it suffices to set $\gamma$ such that 
\begin{equation*}
    \left(\frac{D^{k-1}}{\widetilde{\lambda}}\right)^{\frac{2}{3k+1}}
    \leq
    \frac{10\gamma^{\frac{k-1}{k}}}{4\left(\widetilde{\lambda}+\sqrt{2\widetilde{\lambda}\gamma^{\frac{k-1}{k}}}\right)}+\frac{10}{4}~.
\end{equation*}
By introducing a substitute variable $r:=\gamma^\frac{k-1}{2k}$, the inequality above reduces to a simple quadratic in $r$ which can be easily verified to have a negative discriminant. Consequently, every assignment of $\gamma$ satisfies the desired inequality. In particular, setting $\gamma=1$ gives rise to the conditions
\[
\epsilon\leq\frac{\lambda}{32},~
D\geq\sqrt{\frac{\mu_{k}}{k!2^{\frac{k+3}{2}}\lambda}+\sqrt{2}\left(\frac{\mu_{k}}{k!2^{\frac{k+3}{2}}\lambda}\right)^{3/2}}~,
\]
where the former follows from \eqref{epsilon,d bound low dim}, while the latter follows from the second bullet in Proposition \ref{proposition: minimizer low dim}. This completes the proof.

\end{document}